\documentclass[12pt]{amsart}
\usepackage{amscd,amssymb,subfigure}
\usepackage[all]{xy}
\usepackage{graphicx,verbatim}
\usepackage[colorlinks,plainpages,backref,urlcolor=blue]{hyperref}

\numberwithin{equation}{section}

\theoremstyle{plain}
\newtheorem{theorem}[subsection]{Theorem}

\newtheorem{prop}[subsection]{Proposition}
\newtheorem{corollary}[subsection]{Corollary}

\theoremstyle{definition}

\newtheorem{example}[subsection]{Example}

\newcommand{\A}{\mathcal{A}}
\newcommand{\I}{\mathcal{I}}

\newcommand{\CC}{\mathbb{C}}
\newcommand{\Z}{\mathbb{Z}}
\newcommand{\K}{\mathbb{K}}

\DeclareMathOperator{\gr}{gr}

\begin{document}

\title[On torsion freeness  for the decomposable Orlik-Solomon algebra]%
{On torsion freeness for the decomposable Orlik-Solomon algebra}

\author[A.~M\u{a}cinic]{Anca~M\u{a}cinic$^*$}
\address{Simion Stoilow Institute of Mathematics, 
P.O. Box 1-764,
RO-014700 Bucharest, Romania}
\email{Anca.Macinic@imar.ro}
\thanks{$^*$ Partially supported by 
 a grant of the Romanian Ministry of Education and Research, CNCS - UEFISCDI, project number PN-III-P4-ID-PCE-2020-2798, within PNCDI III}



\date{February 2023}

\begin{abstract}
We prove the torsion freeness of the decomposable Orlik--Solomon algebra of a simple matroid on ground set $[n]$. In the class of hypersolvable  \& non-supersolvable complex hyperplane arrangements, the  torsion freeness, in a certain degree, of this combinatorially defined object, associated to the intersection lattice of the arrangement,  impacts on the first non-vanishing higher homotopy group of the complement of the arrangement.
\end{abstract}

\maketitle

\section{Introduction}
\label{sec:introd}

Let $\A$ be an arrangement of hyperplanes in a complex vector space $V$, i.e., a finite set of hyperplanes in $V$. The {\it intersection lattice} of $\A$, denoted $\mathcal{L}(\A)$, is by definition the poset of intersections of subsets of $\A$, ordered by reverse inclusion, with a rank map given by codimension (see \cite{OT} for a comprehensive approach).  $\mathcal{L}(\A)$ is a geometric lattice, so it translates into the the existence of a simple matroid $M(\A)$ over the ground set $\A$, where dependency is defined as linear dependency. We will refer to  a property or invariant associated to the arrangement as {\it combinatorially determined} (or simply, {\it combinatorial}), if they are determined by the intersection lattice of the arrangement.\\

A strongly combinatorial class of arrangements (in the above sense) is the class of  {\it supersolvable arrangements}, introduced by Stanley in \cite{S}, see also  the equivalent notion of {\it fiber type} arrangements, topologically defined by Falk-Randell in \cite {FR}. 
Supersolvability was generalized by Jambu-Papadima (\cite{JP1}) to {\it hypersolvability}. Hypersolvable arrangements are characterised by the fact that they admit deformations into supersolvable arrangements (see \cite{JP2} for the precise definition of such deformations).
 A notable subclass of hypersolvable arrangements is the class of $2$-generic arrangements (arrangements whose intersection lattice has no dependency in rank $2$). 
Complements of supersolvable arrangements are known to be $K(\pi,1)$ spaces, property that no longer holds in the general hypothesis of hypersolvability. In fact, asphericity of the complement characterizes supersolvability inside the hypersolvable class, so it makes sense to study the occurrence and properties of higher non-trivial homotopy groups of the complement of a hypersolvable, non-supersolvable, arrangement.

This short note  tries to answer a question raised in \cite{MMP} concerning the torsion freeness for the {\it decomposable Orlik-Solomon} (OS) algebra associated to an arrangement of hyperplanes. \\

The motivation comes from a result in \cite{MMP} (see Theorem \ref{thm:MMp_torsion}), which gives a topological meaning to the torsion freeness of the decomposable OS algebra (in a certain degree), for hypersolvable \& non-supersolvable arrangements of hyperplanes,  via a connection to the first non-vanishing higher homotopy group of the complement of the arrangement.\\

In general, the higher homotopy groups of a complement of an arrangement are very difficult to compute, and the same holds true even for their 'approximation' by graded objects (see \cite{PS}). If $\pi_1:=\pi_1(X(\A))$ is the fundamental group of the complement $X(\A)$ of an arrangement $\A$, then any higher homotopy group $\pi_{*>1}:=\pi_{*>1}(X(\A))$ has a module structure over the group ring $R:= \Z \pi_1$, and one can consider the associated graded module 

\begin{equation}
\label{eq:graded_pip}
gr_I^{\bullet}(\pi_p):= \oplus_{q\ge 0} (\pi_p I^q/\pi_p I^{q+1})
\end{equation}

\noindent over the graded ring  $\gr_I^\bullet R$, where $I \subset R$ is the augmentation ideal  (in line with the notations from \cite{MMP}, throughout the text, $-^{\bullet}$ will signify the presence of a grading).\\

For hypersolvable arrangements $\A$ consider the first non-vanishing higher homotopy group $\pi_{p(\A)}$ of the complement $X(\A)$. Then $p(\A)$ is combinatorially determined, see \cite{PS}.
   The graded ring  $\gr_I^\bullet R$ is combinatorial and torsion free (\cite{DP}), and so is the first graded piece $\gr_I^0 \pi_{p(\A)}$ of the associated graded module \eqref{eq:graded_pip} (\cite{PS}).  In \cite{MMP} the attention turns to the second graded piece of the graded module \eqref{eq:graded_pip} associated to the  first non-vanishing higher homotopy group, $\gr ^1_I \pi_{p(\A)}$, whose torsion freeness is characterised in terms of the torsion freeness, in a certain degree, of a combinatorially defined graded algebra, the so called decomposable OS algebra (see the next section for the definition).\\
   
   Our main result, Theorem \ref{thm:main}, proves the torsion freeness of the decomposable OS algebra of an arbitrary simple matroid on ground set $[n]$. In particular, the decomposable OS algebra associated to an arrangement of hyperplanes is torsion free.
As a corollary, we obtain that $\gr ^1_I \pi_{p(\A)}$ is also torsion free.

\section{Preliminaries}

Let $\Lambda^{\bullet}$  be the free exterior algebra over $\Z$ on $n$ generators $\{e_1, \dots, e_n \}$. 
The degree $m$ component is generated by monomials $e_X:=e_{i_1} \wedge  \dots \wedge e_{i_m}$, where $X \subseteq [n]$ is the ordered set $\{ i_1 < \dots < i_m\}$ and $e_{\emptyset} =1$.  
Define the boundary map on  $\Lambda^{\bullet}, \; \partial: \Lambda^{\bullet} \rightarrow \Lambda^{\bullet-1}$ to be the unique degree $-1$ graded algebra derivation such that $\partial(e_i) = 1, \; i  \in [n]$. \\

Let $M$ be a simple matroid on the ground set $[n]$. The {\it Orlik-Solomon ideal} of $M$ is the (graded) ideal $\I(M)$ in $\Lambda^{\bullet}$ generated by the boundaries $\{ \partial (e_C)\; |\; C\ \in \mathcal{C}(M)\}$  of the circuits in $M$.
The {\it decomposable Orlik-Solomon ideal} is the product ideal $\Lambda^{+} \mathcal{I}(M) \subset \mathcal{I}(M)$, where $\Lambda^{+}$ is the ideal of $\Lambda^{\bullet}$ consisting of elements of strictly positive degree.
The {\it Orlik-Solomon algebra} of $M$ is the graded quotient algebra $A^{\bullet}(M) := \Lambda^{\bullet} / \mathcal{I}(M)$. The {\it decomposable Orlik-Solomon algebra}  of $M$ is the graded quotient algebra $A_{+}^{\bullet}(M) := \Lambda^{\bullet} /\Lambda^{+} \mathcal{I}(M)$. 
All the above objects can be defined over an arbitrary field $\K$. As ongoing convention,  we will denote by $-_{\K}$ the specialization to $-\K$ coefficients.  \\

We will be mainly interested in the case of matroids coming from  arrangements of hyperplanes.
 A {\it circuit} $C$ in an arrangement $\A$ is a minimally (linearly) dependent subset of $\A$.  Denote by $\I(\A)$ the OS ideal of $\A$, and by $A^{\bullet}(\A), \; A_{+}^{\bullet}(\A)$ the OS, respectively decomposable OS algebra of $\A$.

 The OS algebra  $A^{\bullet}(\A)$ of an arrangement is torsion free, since it admits a basis expressed solely in terms of circuits of $\A$, that does not depend on the field of coefficients, for any specialization $A^{\bullet}_{\K}(\A)$ (see for instance \cite{F, Y}). It also carries significant topological meaning, as it is isomorphic to the cohomology algebra, with integer coefficients, of the complement of the arrangement, by an emblematic result of Orlik and Solomon. If the arrangement $\A$ is hypersolvable, the decomposable OS algebra of $\A$ reflects some of the topology of the complement:

\begin{theorem}(\cite{MMP})
\label{thm:MMp_torsion}
Let $\A$ be a hypersolvable and not supersolvable complex hyperplane arrangement with complement $X(\A)$, and $p>1$ minimal with the property that $\pi_p(X(\A))$ is non-trivial. Then the following are equivalent:
\begin{enumerate}
\item The second graded piece of the graded module \eqref{eq:graded_pip}, $\gr^1_I \pi_p(X(\A))$, has no torsion.
\item The decomposable Orlik-Solomon algebra,  $A_{+}^{\bullet}(\A)$, is free in degree $\bullet= p+2$.
\end{enumerate}
\end{theorem}

In the hypothesis of the previous theorem, we know that the minimal length of a circuit $C \in \mathcal{C}(\A)$ with  $|C|>3$ is equal to $p+2$. In particular, for a $2$-generic arrangement $\A$,  $p+2$ is the minimal length of a circuit in $\mathcal{C}(\A)$.\\

\section{Torsion freeness for the decomposable OS algebra}

Let $\Lambda^{\bullet}_{\K}$ be the free exterior algebra on $n$ generators $\{e_1, \dots, e_n \}$ over a field $\K$, defined in the previous section. A {\it monomial order} on  $\Lambda^{\bullet}_{\K}$  is a total order on the set of monomials $Mon(\Lambda^{\bullet}_{\K})$ of $\Lambda^{\bullet}_{\K}$ such that:
\begin{enumerate}
\item $1<e_X$ for any monomial $e_X \neq 1$
\item if $e_X<e_Y$ then $e_X \wedge e_Z < e_Y \wedge e_Z$, for $e_X, e_Y, e_Z$ monomials such that $e_X \wedge e_Z \neq 0 \neq e_Y \wedge e_Z$.
\end{enumerate}
 We say that a monomial $e_X$ {\it divides} a monomial $e_Y$ iff $X \subseteq Y$.\\

Let $f = \sum_{i=1}^m a_i e_{X_i} \in \Lambda^{\bullet}_{\K}$, where $a_i \in \K \setminus \{0\}$ and $e_{X_i} \in Mon(\Lambda^{\bullet}_{\K})$. The support of $f$ is the set of monomials  $supp(f):= \{e_{X_1}, \dots, e_{X_m}\}$.
$in_{\prec}(f)$, the {\it initial monomial} of $f$ with respect to a monomial order $\prec$,  is the biggest monomial with respect to $\prec$ among $\{e_{X_1}, \dots, e_{X_m}\}$.

A {\it Gr\"obner basis} for a non-zero ideal $\mathcal{I}$ in the exterior algebra $\Lambda^{\bullet}_{\K}$, with respect to a monomial order $\prec$, is a  finite set of elements $ \mathcal{G}_{\prec} \subset \mathcal{I}$ with the property that for any $f \in \mathcal{I}$ there exists $g \in  \mathcal{G}_{\prec}$ such that $in_{\prec}(g)$ divides  $in_{\prec}(f)$. 
A Gr\"obner basis is called {\it reduced} if, for all $g \in \mathcal{G}_{\prec},\; in_{\prec}(g)$ has coefficient equal to $1$ and, for all $g \neq g' \in \mathcal{G}_{\prec}$, no monomial in $supp(g')$ is divisible by $in_{\prec}(g)$.
We refer to \cite{HH} for more details on the theory of Gr\"obner bases over exterior algebras.\\

In \cite{Y}, respectively \cite{F}, one constructs Gr\"obner bases for OS ideals (for arrangements, respectively for simple matroids in general), which are then used to describe bases for the associated OS algebras, bases that do not depend on the chosen field of coefficients. This already proves torsion freeness for the considered OS algebras.\\

 Let $\mathcal{C}(M)$ denote the set of circuits of a simple matroid $M$ on ground set $[n]$, $\mathcal{C}_q(M): = \{C \in \mathcal{C}(M)\; | \; |C| = q \}$ the subset of circuits of length $q$ and
$$\overline \partial_q : \langle e_C\; |\; C \in \mathcal{C}_{q+1}(M) \rangle_{\K}  \rightarrow (\I(M) / \Lambda^+ \I(M))^q_{\K} $$

\noindent the composition of the restriction $ \partial|_{\langle e_C\; |\; C \in \mathcal{C}_{q+1}(M) \rangle_{\K}}$ to the projection $\I^{q}_{\K}(M) \twoheadrightarrow (\I(M) / \Lambda^+ \I(M))^q_{\K}$, where  $\langle \rangle_{\K}$ denotes $\K$-span. \\

In \cite{MMP} one proves the torsion freeness in all degrees for the decomposable OS algebra associated to a graphic arrangement, by showing that the map $\overline \partial_q$, restricted to chordless $(q+1)$-circuits, is actually an isomorphism. It is not difficult to see, taking into account the description of circuits in signed graphic arrangements given by Zaslavsky (\cite{Z}), that the same technique can be used to prove the torsion freeness in all degrees for the decomposable OS algebra associated to a signed graphic arrangement.
The problem with this approach is that it does not extend to general arrangements, as shown by \cite[Ex. 4.7, Ex. 4.8]{MMP}.

We solve this problem by replacing the set of chordless $(q+1)$-circuits by a more suitable set of $(q+1)$-circuits, so the restriction of the map $\overline \partial_q$ remains an isomorphism.
We essentially use a result of Forge (\cite{F}) that describes the reduced Grobner basis of the OS ideal, for an arbitrary monomial order.\\

We recall some notions related to matroids, following \cite{F}. Let $M$ be a simple matroid on the ground set $[n]$, and $I$ an independent set of $M$. An element  $i \in [n]$ is called {\it active} with respect to $I$ if $I \cup \{i\}$ contains a circuit with minimal element $i$. 

Any permutation $\pi$ of $[n]$ defines an order on the elements of $[n]$,
$$ \; \pi^{-1}(1) <_{\pi} \pi^{-1}(2) <_{\pi} \dots <_{\pi} \pi^{-1}(n), $$ and a monomial order $\prec_{\pi}$ induced by 
$$e_{\pi^{-1}(1)} \prec_{\pi} e_{\pi^{-1}(2)} \prec_{\pi} \dots \prec_{\pi} e_{\pi^{-1}(n)}.$$

\noindent Let $\mathcal{C}^{\pi}(M)$ be the subset of circuits of $M$ such that:
\begin{enumerate}
\item $inf_{\prec_{\pi}}(C) = \alpha_{\pi}(C)$
\item  $C \setminus \alpha_{\pi}(C)$ is inclusion minimal with property (1), 
\end{enumerate}
where  $\alpha_{\pi}(C)$ is by definition the smallest active element with respect to the independent $C \setminus \{inf_{\prec_{\pi}}(C)\}$, relative to the order $\prec_{\pi}$.

\begin{theorem}(\cite{F})
\label{thm:forge}
 $\mathcal{G}_{\pi}:= \{ \partial(e_C) | C \in \mathcal{C}^{\pi}(M)\}$ is a reduced Gr\"obner basis for the ideal $\mathcal{I}(M)$ with respect to the monomial order $\prec_{\pi}$.
\end{theorem}

\begin{prop}
\label{prop:partial_bar}
Let $\pi$ be a permutation such that 
$$|\mathcal{C}_{q+1}^{\pi}(M)| = min\{|\mathcal{C}_{q+1}^{\pi'}(M)|\; |\; \pi' permutation \; of\; [n]\}.$$
 Then the map 
$$\overline\partial_q : \langle e_C \; | \;  C \in \mathcal{C}^{\pi}_{q+1}(M) \rangle _\K \rightarrow (\I (M)/ \Lambda^+ \I(M))_\K^q$$
 is an isomorphism.
\end{prop}

\begin{proof}
Since $\mathcal{G}_{\pi}$ is a Gr\"obner basis for $\I(M)$, it's also a system of generators for $\I(M)$, so it immediately follows that the map $\overline\partial_q$ is a surjection. 
Also, notice that the map $\partial_q :  \langle e_C \; | \;  C \in \mathcal{C}^{\pi}_{q+1}(M) \rangle _\K \rightarrow \I (M)_\K^q$  is injective, following from the fact that  $\mathcal{G}_{\pi}$ is reduced.

To prove injectivity for  $\overline\partial_q$, let 
\begin{equation}
\label{eq:sum}
\mathfrak{e} = \sum_{C \in  \mathcal{C}^{\pi}_{q+1}(M)}\mu_C e_C  \in Ker (\overline \partial_q), \; \mu_C \in \K, \; \overline \partial_q(\mathfrak{e}) = 0.
\end{equation}
 Then
\begin{equation}
\label{eq:*}
\sum_{C \in  \mathcal{C}^{\pi}_{q+1}(M)}\mu_C \partial (e_C) =  \sum_{S, \overline{C}}\xi_{\overline{C}, S} e_S \partial(e_{\overline{C}}), \; \xi_{\overline{C}, S} \in \K, \; e_S \in \Lambda^+, \; \overline{C} \in \mathcal{C}^{\pi}_{\leq q}(M).
\end{equation}

Consider the initial monomials $in_{\prec_{\pi}}(\partial(e_C))$ for all $C$ appearing in the sum \eqref{eq:sum}. All these monomials must be distinct, otherwise, there would exist two $(q+1)$-circuits, say $C_1, C_2 \in \mathcal{C}^{\pi}_{q+1}(M)$, such that $in_{\prec_{\pi}}(\partial(e_{C_1})) = in_{\prec_{\pi}}(\partial(e_{C_2}))$, then $C_1$ and $C_2$ would differ by an element, $inf_{\prec_{\pi}}(C_1) \neq inf_{\prec_{\pi}}(C_2)$. But this implies that  $C_1, C_2 $ cannot be simultaneously elements in $\mathcal{C}^{\pi}(M)$, since either $\alpha_{\pi}(C_1) \prec_{\pi} inf_{\prec_{\pi}}(C_1) $ or  $ \alpha_{\pi}(C_2) \prec_{\pi} inf_{\prec_{\pi}}(C_2)$, contradiction.

In fact, since $C \in  \mathcal{C}^{\pi}_{q+1}(M)$ for all $(q+1)$-circuits $C$ in \eqref{eq:sum}, i.e. $\partial(e_C)$ are all elements of a reduced Gr\"obner basis for $\I(M)$, the monomials $in_{\prec_{\pi}}(e_C)$ do not cancel out with any other monomial in the left hand side of the equality \eqref{eq:*}. So, if the coefficient $\mu_C \neq 0$, then on the right side of  \eqref{eq:*} there exists $\overline{C} \in \mathcal{C}^{\pi}_{\leq q}(M)$ such that $in_{\prec_{\pi}}(\partial(e_C))$ is a monomial in $e_S \partial(e_{\overline{C}})$, for some $e_S \in \Lambda^+, \; \overline{C} \in \mathcal{C}^{\pi}_{\leq q}(M)$. 
Say $e_X$ is the monomial in $\partial(e_{\overline{C}})$ such that $  in_{\prec_{\pi}}(e_C)= e_S e_X$. 

If $e_X =  in_{\prec_{\pi}}(e_{\overline{C}})$, then $ in_{\prec_{\pi}}(e_{\overline{C}})$ divides  $ in_{\prec_{\pi}}(e_C)$, contradiction, since both $\partial(e_C)$ and $\partial(e_{\overline{C}})$ are elements in the reduced Grobner basis $\mathcal{G}_{\pi}$. 
Then necessarily $e_X \neq in_{\prec_{\pi}}(e_{\overline{C}})$. Let $e_{\overline{C}} = e_{i_1} \wedge  \dots \wedge e_{i_p}, \;  e_{i_1} \prec  \dots  \prec e_{i_p}, \; p \leq q$ 
and $e_X = e_{\overline{C} \setminus \{i_j \}} = e_{i_1} \wedge \dots \widehat{e_{i_j}} \dots   \wedge e_{i_p}$, with $e_{i_1} \prec_{\pi} e_{i_j}$.
In this case, we will show that there is another monomial order, induced by a permutation $\pi_0$ on $[n]$, such that $|\mathcal{C}_{q+1}^{\pi_0}(M)
| < |\mathcal{C}_{q+1}^{\pi}(M)|$, thus arriving at a contradiction.

Take a permutation $\pi'$ on $[n]$ such that $\pi'^{-1}(k) = \pi^{-1}(k)$ for $k <  \pi(i_1)$ or $k >  \pi(i_j)$, $\pi'^{-1}(\pi(i_1)) =  i_j$ and $\pi'^{-1}(k) = \pi^{-1}(k-1)$, for $\pi(i_1) < k \leq   \pi(i_j)$. Basically, the corresponding order $<_{\pi'}$ on $[n]$ is obtained from the order  $<_{\pi}$ by moving $i_j$ immediately to the left of $i_1$.
 
In order to prove $\overline{C} \in \mathcal{C}^{\pi'}(M)$, it is enough to show that 
  $i_j = \alpha_{\pi'}(\overline{C})$.
Take $e_m \in \Lambda^+, m \in [n]$, such that $(m, i_1, i_2, \dots, \hat{i_j}, \dots, i_p)$ contains a circuit with minimal element $m$, in the ordering $\prec_{\pi'}$. Then $m$ is in the closure $cl(\{ i_1, i_2, \dots, \hat{i_j}, \dots, i_p\}) = cl(\overline{C}) = cl(\{ i_2, i_3, \dots, i_j, \dots,  i_p \})$. This means that the set $\{ m, i_2, i_3, \dots, i_j, \dots, i_p \}$ contains a circuit. If moreover we assume $m <_{\pi'} i_j$, then $m <_{\pi} i_1$, contradiction to $i_1 = \alpha_{\pi}(\overline{C})$. It follows that $i_j <_{\pi'} m$, consequently $i_j =  \alpha_{\pi'}(\overline{C})$.

Notice that  $\alpha_{\pi}(C) = \alpha_{\pi'}(C) = e_{inf_{\prec_{\pi}}(C)}$. 
But, since $in_{\prec_{\pi'}}(e_C) = e_S \wedge in_{\prec_{\pi'}}(e_{\overline{C}}) $, $C \notin  \mathcal{C}_{q+1}^{\pi'}(M)$.

We will show that any circuit in  $\mathcal{C}_{q+1}^{\pi'}(M)$, or, eventually,  any circuit in  $\mathcal{C}_{q+1}^{\pi_0}(M)$ (corresponding to a suitable modification $\pi_0$ of $\pi'$) induces a circuit in $\mathcal{C}_{q+1}^{\pi}(M)$, 
 and the correspondence is injective, and, moreover, $\pi_0$ still retains the properties of $\pi'$ above mentioned, related to the circuits $C, \overline{C}$. This injective correspondence, along with the fact that $C \in \mathcal{C}_{q+1}^{\pi}(M)$ is not the correspondent of any circuit in $ \mathcal{C}_{q+1}^{\pi_0}(M)$, would imply that $|\mathcal{C}_{q+1}^{\pi_0}(M)| < |\mathcal{C}_{q+1}^{\pi}(M)|$, contradiction to the hypothesis on the permutation $\pi$.

It is not hard to check that any circuit $C'$ in $\mathcal{C}_{q+1}^{\pi'}(M)$ such that $inf_{\prec_{\pi'}}(C') \neq i_j$ is also a circuit in $\mathcal{C}_{q+1}^{\pi}(M)$, and $inf_{\prec_{\pi'}}(C') = inf_{\prec_{\pi}}(C')$. If a circuit $C_0 \in \mathcal{C}_{q+1}^{\pi'}(M)$ is such that $inf_{\prec_{\pi'}}(C_0)) = i_j$, then, if $C_0 = (i_j, a_1, \dots, a_q)$ is not already a circuit in  $\mathcal{C}_{q+1}^{\pi}(M)$, there are two situations to consider. Notice that if $C_0 \notin \mathcal{C}_{q+1}^{\pi}(M)$, then $i_j$ is no longer the minimal element of $C_0$, with respect to the order induced by $\pi$.

Firstly, if there exists $a' <_{\pi} a_1$, such that $(a', i_j, \widehat{a_1}, a_2, \dots, a_q)$ contains a circuit $c$ with $a' = \alpha_{\prec_{\pi}}(c)$. Then $a' <_{\pi'} a_1$ and the set $S := (a', a_1, \dots, a_q)$ is dependent. But, since $C_0 = (i_j, a_1, \dots, a_q)$ is a circuit in  $\mathcal{C}_{q+1}^{\pi'}(M)$, the set $S$ must be a circuit in itself, otherwise we get a strict subset of $S$ that is a circuit in  $\mathcal{C}_{\leq q}^{\pi'}(M)$ with minimal element $a'$,
a contradiction to the fact that $C_0 \in \mathcal{G}_{\pi'}$, which is reduced. Then, corresponding to the circuit  $C_0 \in \mathcal{C}_{q+1}^{\pi'}(M)$ we get the $(q+1)$-circuit $S$ in $\mathcal{C}_{q+1}^{\pi}(M)$.

Secondly, assume there exists $a' >_{\pi} a_1$, such that $(a', i_j, \widehat{a_1}, a_2 \dots, a_q)$ contains a circuit with $a'$ minimal element with respect to $<_{\pi}$. Then  $a' >_{\pi'} a_1$ and it is possible that $(a', i_j, \widehat{a_1}, a_2 \dots, a_q)$ contains a circuit with length strictly smaller than $q+1$. In this case we alter the permutation $\pi'$ by moving $a'$ immediately to the left of $a_1$, and we are back in the situation from the first case, so we can define a $(q+1)$-circuit associated to $C_0$ in $\mathcal{C}_{q+1}^{\pi}(M)$. The newly obtained permutation, say $\pi''$, still retains the properties required for $\pi'$, that is, $\overline{C} \in \mathcal{C}^{\pi''}(M)$  and $in_{\prec_{\pi'}}(e_C) = e_S \wedge in_{\prec_{\pi'}}(e_{\overline{C}}) $, so $C \notin  \mathcal{C}^{\pi''}(M)$. Moreover, each new modification, in the sense described before, of a permutation is situated to the right of the previous one in the order induced by the permutation. 
Finally, notice that two distinct circuits in $\mathcal{C}_{q+1}^{\pi''}(M)$ cannot correspond to the same circuit in $\mathcal{C}_{q+1}^{\pi}(M)$, due to the fact that $\mathcal{G}_{\pi''}$ is reduced.

In this way, eventually after a number of moves as described before, we get a permutation $\pi_0$ of $[n]$ with the required properties.

\end{proof}

\begin{example}
\label{ex:2orders}
Let us revisit for instance \cite[Example 4.8]{MMP}. Let $\A \subset \CC ^4$ be the arrangement of equation $xyzt(x+y+z+t)(x-y-z+t)=0$. Denote by $H, P$, the hyperplanes of equations $x+y+z+t=0$, respectively  $x-y-z+t=0$. The set of circuits of $\A$ is $\{ (H,P,y,z), \; (H,P,x,t), \; (H,x,y,z,t), \; (P,x,y,z,t) \}$.
Let  $\Lambda^{\bullet}(\A): = \Lambda (e_H, e_P, e_x, e_y, e_z, e_t)$ be the exterior algebra  with generators labelled by the hyperplanes of $\A$. Consider first the monomial order on    $\Lambda^{\bullet}(\A)$ induced by $x<y<z<t<H<P$.  The reduced Gr\"obner basis of the OS ideal $\I(\A)$ with respect to this order is given by the boundaries of the set of  $4$ circuits $\{ (y, z, H, P), \; (x, t, H, P), \; (x, y, z ,t , H),  (x, y, z, t,  P)\}$.

But for the monomial order induced by, say, $H < P < x < y < z < t$, the associated reduced Gr\"obner basis of  $\I(\A)$ is given by the boundaries of the set of $3$ circuits $\{ (H,P,y,z), \; (H,P,x,t), \; (H,x,y,z,t)\}$, and $3$ is the minimal cardinal for the reduced Gr\"obner bases of $\I(\A)$ obtained by varying the order among the hyperplanes of $\A$.

In this example, $dim_{\K}  (\I (\A)/ \Lambda^+ \I(A))_\K^4 = 1$, $(\I (\A)/ \Lambda^+ \I(A))_\K^q = \I (\A)_\K^q$, for $q<4$, and  $(\I (\A)/ \Lambda^+ \I(A))_\K^q = 0$, for $q>4$.
\end{example}

\begin{corollary}
\label{cor:decompI}
The graded abelian group $\I (M)/ \Lambda^+ \I(M)$ is torsion free, in any degree $q$.
\end{corollary}

\begin{proof}
Follows immediately from Proposition \ref{prop:partial_bar}.
\end{proof}

\begin{theorem}
\label{thm:main}
 The decomposable Orlik-Solomon algebra $A_{+}^{\bullet}(M)$  is torsion free.
\end{theorem}

\begin{proof}
It follows from the previous Corollary \ref{cor:decompI} and the split exact sequence:
$$
0 \rightarrow (\I(M) /\Lambda^+ \I(M))^{\bullet} \rightarrow A_{+}^{\bullet} (M) \rightarrow A^{\bullet} (M) \rightarrow 0.
$$
 \end{proof}

\begin{corollary}
\label{cor:pi_p}
In the hypothesis of Theorem \ref{thm:MMp_torsion},  $\gr^1_I \pi_p(X(\A))$ is torsion free.
\end{corollary}

\begin{proof}
Immediately from Theorem \ref{thm:MMp_torsion} and Theorem \ref{thm:main}.
\end{proof}


\begin{thebibliography}{00}

\bibitem{DP} A. Dimca, S. Papadima,
{\em Hypersurface complements, Milnor fibers and higher
homotopy groups of arrangements}, Ann. Math. \textbf{158} (2003), 473--507.

\bibitem{FR} M. ~Falk, R. ~Randell
{\em The lower central series of a fiber-type arrangement},
Invent. Math., 82 (1985),  77--88.

\bibitem{F} D. ~Forge,
{\em Bases in Orlik–Solomon Type Algebras},
Europ. J. Combinatorics (2002) 23, 567--572.

\bibitem{HH}J.~Herzog, T.~Hibi,
{\em Monomial ideals}
Graduate Texts in Mathematics (GTM, volume 260), 2011.

\bibitem{JP1} M.~Jambu, S.~Papadima,
{\em A generalization of fiber-type arrangements and
a new deformation method}, Topology \textbf{37}
(1998), 1135--1164.

\bibitem{JP2} \bysame, {\em Deformations of hypersolvable
arrangements},
Topology Appl. \textbf{118} (2002), 103--111.

\bibitem{MMP} D. ~Matei, A. ~Macinic, S. ~Papadima
{\em On the second nilpotent quotient of higher homotopy groups,
fir hypersolvable arrangements}
Int. Math. Res. Notices vol. 2015, no.24 (2015), 13194--13207.

\bibitem{OT} P.~Orlik, H.~Terao,
{\em Arrangements of hyperplanes}, Grundlehren Math. Wiss., vol.~300,
Springer-Verlag, Berlin, 1992.

\bibitem{PS} S.~Papadima, A.~Suciu, 
{\em Higher homotopy groups of complements of complex hyperplane
arrangements},
Advances in Math. \textbf{165} (2002), 71--100.

\bibitem{S} R. ~Stanley
{\em Supersolvable lattices},
Algebra Universalis, 2 (1972), 197--217.

\bibitem{Y} S.~Yuzvinsky,
{\em Orlik--Solomon algebras in algebra and topology},
Russian Math. Surveys \textbf{56} (2001), 87--166.

\bibitem{Z}T.~ Zaslavsky
{\em Signed graphs},
Discrete Appl. Math.,  \textbf{4} (1) (1982), 47–-74.

\end{thebibliography}
\end{document}